\documentclass[12pt]{amsart}
\RequirePackage[top=2.5cm,bottom=2.5cm,inner=2.2cm,outer=2.2cm,
headsep=1cm,footskip=1cm]{geometry}
\usepackage{amsthm,xcolor}
\usepackage{mathtools}
\usepackage{enumitem}
\usepackage[all,pdf]{xy}
\usepackage{tikz-cd}

\newtheorem{theorem}{Theorem}[section]
\newtheorem{proposition}[theorem]{Proposition}

\newtheorem{lemma}[theorem]{Lemma}
\theoremstyle{definition}
\newtheorem{definition}{Definition}[section]
\newtheorem{remark}[theorem]{Remark}
\newtheorem{notation}[theorem]{Notation}
\newtheorem{example}[theorem]{Example}

\RequirePackage{xr-hyper}
\RequirePackage{hyperref}

\hypersetup{bookmarks=true}%
\hypersetup{bookmarksopen=true}%

\RequirePackage{amsmath,amssymb,amscd}
\newcommand{\Img}[1]{\ensuremath{\mathrm{Im} (#1)}}

\DeclareMathOperator{\Ima}{Im}

\begin{document}
\title{Notes on Super Projective Modules}
\author{Archana~S.~Morye, Aditya~Sarma~Phukon, Devichandrika~V.}
\address{University of Hydrabad,  UW Madison, University of Hyderbad}
\email{asmsm@uohyd.ac.in, sarmaphukon@wisc.edu, 18mmpp03@uohyd.ac.in}
  
\begin{abstract}
Projective modules are a link between geometry and algebra as established by the theorem of Serre-Swan. In this paper, we define the super analog of projective modules and explore this link in the case of some particular super geometric objects. We consider the tangent bundle over the supersphere and show that the module of vector field over a supersphere is a super projective module over the ring of supersmooth functions. Also, we discuss a class of super projective modules that can be constructed from a projection map on modules defined over the ring of supersmooth functions over superspheres.
\end{abstract}
\maketitle
\tableofcontents

\section{Introduction}
\label{section:01}
\textit{Projective modules} were first introduced by Cartan and Eilenberg in their 1956 Homological Algebra book \cite{CE1956} as a generalization of free modules. They defined these as modules $P$ such that if given any module homomorphism $f:P\rightarrow M''$ and any module epimorphism $g:A\rightarrow M''$, there is a module homomorphism $h:P \rightarrow M$ with $gh =f$. 

Shortly after, several nice properties of projective modules were discovered. Countably generated projective modules (and hence their direct sum) over a local ring were shown to be free modules by Kaplansky in 1958 \cite{Kap1958}. It followed from this that the category of locally free modules (over suitable ringed spaces) was equivalent to that of projective modules over global rings. With this, vector bundles could now be related to projective modules. The categorical equivalence of this was given by the Serre-Swan theorem \cite{Ser1955, Swa1962, Mor2013}, essentially bridging Algebra and Geometry. Later, with the work of Govorov and Lazard, it was shown that a module with finite presentation is flat if and only if projective, which laid the definitive link between flat, projective and free modules. 

\textit{Supersymmetry}, and subsequently, the bigger field of supermathematics, draws motivation from Physics. It is used to derive a potential extension to the standard model, proposing that for each Fermion, there exists a supersymmetric Boson and vice versa. A proper mathematical setup of supersymmetry was made available when Berezin and Katz \cite{BK1970} introduced supergroups, which were groups with commutating and anti-commutating parameters, and then partly based on the work by Grassmann, \textit{graded Lie algebra }(or Lie superalgebra). The latter was of interest to physicists as commutative Bosonic and anti-commutative Fermionic ladder operators could be encoded into a Lie superalgebra. Later, in a paper by Salam and Strathdee \cite{SS1974} a new `space' was introduced as an analog to Minkowski by considering a coset of a particular supergroup (Poincar\'{e} supergroup quotiented by Lorentz supergroup). This was named \textit{superspace} which was later used to define \textit{supermanifolds}. A more rigorous treatment of this space was done by Rogers \cite{Rog1986}, Boyer et al \cite{BG1984} and some others. This involved developing the analysis on \textit{supermanifolds} and introducing supersmooth functions. More super analogs of non-super geometric objects were subsequently introduced of which graded or supermodules defined by Jadczyk \cite{JP1981} are of particular interest. 

In this study, we define an analog of projective modules in superalgebra. The motivation for this comes from the question on whether the analog of the Serre-Swan theorem holds where the objects are replaced by their super versions. 
The paper is organized as follows. Section \ref{section:02} introduces the super analogs for objects - \emph{vector space, rings, and modules}. We first recall the definitions of a \textit{super vector space} and a \textit{super commutative ring} (Definition \ref{def:2.1}), building to the idea of a \textit{superalgebra} (Definition \ref{def:2.2}). We then discuss about \textit{supermodules} (Definition \ref{def:2.7}) and then subsequently \textit{free supermodules}. We also verify that the universal property of free supermodules holds (Theorem \ref{theorem:2.12}).

In Section \ref{section:03}, we introduce the main object of study -- \textit{super projective modules} and define them by giving three equivalent conditions (Theorem \ref{theorem:3.1}). This theorem can be used to find different examples of super projective modules such as those constructed by taking a supermanifold $X$ and then considering $C(X)-$modules. We also provide Proposition \ref{proposition:3.4} that show super projectivity of modules is preserved under tensor product and direct sum and that *End* of super projective modules is also super projective. We then proceed to give basic examples constructed using the established techniques in this section.

In Section \ref{section:04} we build up to the super analog of stably free modules. After defining a supercircle and verifying that it is a supermanifold, we show that the tangent bundle of the supercircle is trivial. Using this fact, we construct a stably free supermodule. The super version of this is included in the section, inspired by Kaplansky's proof which used the result that a free module can be decomposed into the kernel and image of an idempotent function. We go a dimension up and introduce superspheres. We then show two classes of super projective modules that can be obtained from them - one using tangent bundles and the other using projectors defined over the superspheres. The latter has been studied in \cite{Lan2001}.

\section {Superalgebra} 
\label{section:02}

Grassmann introduced the idea of superalgebra in 1844. Though the idea was not initially appreciated, in the second half of the 20th century, due to its use in differential geometry (differential forms, de Rham cohomology), and supersymmetry the idea of superalgebra is now extensively used in mathematics as well as in physics.  

In order to make the paper self-contained, we will recall basic notations and terminologies from superalgebra in this section. Towards this, we start the section with basic definitions such as super vector spaces, super commutative rings, and supermodules over them. We discuss free supermodules elaborately in this section and give proof of the universal property of the free supermodules (Theorem \ref{theorem:2.12}). For more details on superalgebra one can refer to  \cite[pp.~8-19]{Rog2007}, or \cite[Chapter 3, Section 1 to 6, pp.154-157]{Man1984}. 

\subsection{Super commutative rings}
\label{subsection:2.1}

A ring $(R,+,\cdot)$ is said to be a \emph{superring} if $(R,+)$ has two subgroups $R_0$ and $R_1$, such that $R=R_0\oplus R_1$, and $R_\alpha\cdot R_\beta\subset R_{\alpha+\beta}$ for all $\alpha,\beta\in \mathbb{Z}_2$.  We assume a ring with identity, and  $1\in R_0$. If $a\in R_0$ or $a\in R_1$ then $a$ is called a \emph{homogeneous element} of $R$, and we denote by $\lvert a\rvert=0$ (respectively $\lvert a\rvert=1$), if $a\in R_0$ (respectively $a\in R_1$). A superring which is also a field is called \emph{superfield}. 

Suppose $V$ is vector space over a superfield and $V_{0}, V_{1}$ are two subspaces such that $V=V_{0}\oplus V_{1}$. Then $V$ is called a \emph{super vector space}. 
The elements of $V_{0}$ are called \emph{even} and those of $V_{1}$ are called \emph{odd}. Any element $v \in V_i$ is called \emph{homogeneous} with degree, $\lvert v\rvert = i$ if $v\in V_{i}$, $i\in \mathbb{Z}_{2}$. 

\begin{definition} 
\label{def:2.1}
A superring $R$ is called \emph{super commutative ring} if 
$$[a,b]=ab-(-1)^{\lvert a \rvert \lvert b\rvert}ba=0 $$  
for all $a$, $b$ homogeneous elements of $R$.
\end{definition}

\begin{definition}
\label{def:2.2}	
Let $A=A_{0}\oplus A_{1}$ be a super vector space such that $A$ is also an algebra over a field $K$. A super vector space $A$ is a \emph{superalgebra} if it is also a superring. 
A superalgebra $A$ which is a super commutative ring is called  \emph{super commutative algebra}.
\end{definition}

\begin{notation}
\label{nota:2.3}
The set of natural number $\{0, 1,2,\ldots\}$  is denoted by $\mathbb{N}$.  Let $L, k \in \mathbb{N}$. If $L, k\neq 0$, we denote by $\underline{\lambda}$, a multi index $\underline{\lambda} = \lambda_1 \ldots \lambda_k$ with $1 \leq \lambda_1 < \lambda_2 < \cdots < \lambda_k \leq L$. We denote the empty multi index by $\underline{0}$ (corresponds to $k=0$). When $L=0$, by convention we will assume that there is only one multi index $\underline{0}$. Let $M_{[L]}$ be the set of all such multi indices. In particular $M_{[0]}=\{\underline{0}\}$. Let $l(\underline{\lambda})$ denote the number of indices in multi index $\underline{\lambda}$.  So $l(\underline{\lambda})=k$, for $\underline{\lambda} = \lambda_1 \ldots \lambda_k$, and $l(\underline{0})=0$. We denote by $M_{[L,0]}$ (respectively, $M_{[L,1]}$)the subset of $M_{[L]}$ consists of all multi index $\underline{\lambda}$ with $l(\underline{\lambda})$ is even (respectively, odd) length. Let $M_{[\infty]}$ define the set of multi indices  $$M_{[\infty]}=\cup_{L\in\mathbb{N}}M_{[L]}=\{\underline{\lambda}=\lambda_{1}\lambda_{2}\dots \lambda_{k}\vert 1 \leq \lambda_{1} < \lambda_{2} < \cdots < \lambda_{k}, \text{ for } k\in\mathbb{N}>0\}\cup \{\underline{0}\}.$$ Further,  $M_0=\{\underline{\lambda} \in M_{[\infty]}\vert l(\underline{\lambda})\text{ is even}\}$ and $M_1=\{\underline{\lambda}\in M_{[\infty]}\vert l(\underline{\lambda})\text{ is odd}\}$, hence $M_{[\infty]}=M_{0}\cup M_{1}$. 
\end{notation}

\begin{example}
\label{example:2.4}
For each finite positive integer $L$, $\mathbb{R}_{S[L]}$ denotes the \emph{Grassmann algebra over the ring of real numbers $\mathbb{R}$} with $L$ generators $\beta_{[0]}=1, \beta_{[1]} , \beta_{[2]} , \ldots ,\beta_{[L]}$, that is, it satisfies the relations 
\begin{equation}
\begin{gathered}
1 \beta_{[i]} = \beta_{[i]} = \beta_{[i]} 1, \ \text{for} \  i = 1,2, \ldots , L, \\
\hspace{.6cm}\beta_{[i]} \beta_{[j]} = - \beta_{[j]} \beta_{[i]}\ \text{for} \ i,j=1,2,\ldots, L.
\end{gathered}
\end{equation}
Any element $x$ of $\mathbb{R}_{S[L]}$ may thus be expressed as
$$x = \sum _{\underline{\lambda} \in M_{[L]}} x_{\underline{\lambda}} \beta_{[\underline{\lambda}]},$$
where $x_{\underline{\lambda}}$ is a real number, and $\beta_{[\underline{\lambda}]} = \beta_{[\lambda_1]} \cdots \beta_{[\lambda_k]}$, (with $\beta_{[\underline{0}]} = 1$) for all ${\underline{\lambda} \in M_L}$.
	
The Grassmann algebra $\mathbb{R}_{S[L]}$ is given the structure of a super commutative algebra by setting
$$\mathbb{R}_{S[L]} = \mathbb{R}_{S[L,0]} \oplus \mathbb{R}_{S[L,1]}, $$
with $\mathbb{R}_{S[L,0]}=\{x\in \mathbb{R}_{S[L]}\vert x =\sum_{\underline{\lambda} \in M_{[L,0]}} x_{\underline{\lambda}}\beta_{[\underline{\lambda}]}\}$, and $\mathbb{R}_{S[L,1]}=\{x\in \mathbb{R}_{S[L]}\vert x=\sum_{\underline{\lambda} \in M_{[L,1]}} x_{\underline{\lambda}}\beta_{[\underline{\lambda}]}\}$.
\end{example}
Similarly we can construct the Grassmann algebra over $\mathbb{R}$ with infinite generators as follows:

\begin{example}
\label{example:2.5}	
Let $\mathbb{R}_{S}$ be the algebra over $\mathbb{R}$ with generators $\{\beta_{[k]} \vert k \in \mathbb{N} \}$ where $\beta_{[0]} = 1$ and relations
\begin{equation}
\begin{gathered}
1 \beta_{[i]} =\beta_{[i]}=\beta_{[i]} 1 \qquad i\in \mathbb{N} \\
\beta_{[i]} \beta_{[j]} =-\beta_{[j]} \beta_{[i]} \qquad i, j\in \mathbb{N}\setminus \{0\}.
\end{gathered}
\end{equation}
Then any element $x \in \mathbb{R}_{S}$ is a formal sum
$$x = \sum_{\underline{\lambda}=\lambda_{1}\lambda_{2}\dots \lambda_{k} \in M_{[\infty]}}x_{\underline{\lambda}}\beta_{[\lambda_{1}]}\beta_{[\lambda_{2}]}\dots\beta_{[\lambda_{k}]},$$
where each $x_{\underline{\lambda}}\in \mathbb{R}$.
We can give $\mathbb{R}_{S}$ a superalgebra structure by defining the even elements and odd elements as follows:
Consider the subspaces $\mathbb{R}_{S_{0}}$ and $\mathbb{R}_{S_{1}}$ defined as 
\begin{align*}
\mathbb{R}_{S_{0}}&=\left\{x \in \mathbb{R}_{S} \bigg\vert x=\sum_{\underline{\lambda}=\lambda_{1}\lambda_{2}\dots \lambda_{k} \in M_{0}} x_{\underline{\lambda}}\beta_{[\lambda_{1}]}\beta_{[\lambda_{2}]}\dots\beta_{[\lambda_{k}]} \right \}\\
\mathbb{R}_{S_{1}}&=\left\{\xi\in \mathbb{R}_{S} \bigg\vert \xi=\sum_{\underline{\lambda}=\lambda_{1}\lambda_{2}\dots \lambda_{k} \in M_{1}} \xi_{\underline{\lambda}}\beta_{[\lambda_{1}]}\beta_{[\lambda_{2}]}\dots\beta_{[\lambda_{k}]}\right\}.
\end{align*}
Then, clearly $\mathbb{R}_{S}=\mathbb{R}_{S_{0}}\oplus \mathbb{R}_{S_{1}}$ making it a superalgebra.
\end{example}

Most properties of $\mathbb{R}_{S[L]}$ are also true for $\mathbb{R}_{S}$. However, there are some exceptions. One example is characterization of nilpotent elements of $\mathbb{R}_{S[L]}$; an element $x$ is nilpotent if and only if $x_{\underline{0}}=0$.  This is due to the fact that if $x_{\underline{0}} = 0$, then $x^{L + 1} = 0$. But this result is not true for $\mathbb{R}_S$. 

\begin{example}
\label{example:2.6}
Let $x = \Sigma _{i < j} \beta _{[i]} \beta_{[j]} \in \mathbb{R}_S$. Clearly $x_{\underline{0}} = 0$ but $x^n \neq 0$. When $n =1$, the coefficient of $\beta _{[1]} \beta_{[2]} = 1$. For $n = 2$, the coefficient of $\beta _{[1]} \beta_{[2]} \beta _{[3]} \beta_{[4]}$ is $2$. Let $\lambda_n = 1 \leq 2 < \cdots \leq 2n$, by induction we can see that the coefficient of $\beta_{[\lambda_n]}$ in $x^n$ is $n!$.  Hence $x^n\neq 0$, for all $n\in\mathbb{N}$, in particular $x$ is not nilpotent.    
\end{example}

\subsection{Supermodules}
\label{subsection:2.2}
Superrings are non-commutative; hence we need to define both the left and right supermodules. However, if the ring is super commutative then by simple sign change, we can convert the left module into right and vice versa. This makes the category of supermodules over super commutative rings comparatively simple.  In this subsection, we will define the category of supermodules over a super commutative ring. 

\begin{definition}
\label{def:2.7}
Let $R$ be a superring, and $M$ a right (respectively, left) module over $R$. A module $M$ is a \emph{right} (respectively, \emph{left}) \emph{supermodule over $R$} if it has two submodules $M_0$ and $M_1$ such that $M=M_0\oplus M_1$, and if for all $\alpha, \beta\in \mathbb{Z}_2$, one has $ M_\beta R_\alpha\subset M_{\beta+\alpha}$ (respectively, $R_\alpha M_\beta\subset M_{\alpha+\beta}$).
\end{definition}

Let $R$ be a super commutative ring and $M$ be any right supermodule over $R$.  Define $ax=(-1)^{\mid x \mid \mid a \mid} xa$ for all $x\in M$, $a\in R$ homogeneous elements.  This gives $M$ a left module structure. Hence, $M$ can be considered both sided supermodule over super commutative ring $R$. In this paper, we are considering all rings to be super commutative. Hence, by the above observation, we will use the term \emph{supermodule over $R$} without mentioning right or left. Here onward we will consider modules as a right module over $R$, unless stated otherwise.

\begin{definition}
\label{def:2.8}
Let $R$ be a super commutative ring, and $M$, $N$ supermodules over $R$. We say that a morphism $\phi:M\to N$ is a \emph{$R$-linear morphism} if $\phi(xa) = \phi(x)a$ for all $x\in M$ and $a\in R$. We say that $\phi$ has \emph{degree} $\vert\phi\vert=\beta\in\mathbb{Z}_2$ if $\phi(M_\alpha)\subset N_{\alpha+\beta}$ for all $\alpha\in \mathbb{Z}_2$. If $\vert\phi\vert = 0$, then $\phi$ is said to be \emph{even morphism} and if $\vert\phi\vert = 1$, then $\phi$ is said to be \emph{odd morphism}.  If $\phi$ is even or odd, we say $\phi$ is \emph{homogeneous}.  
\end{definition}

Since $R$ is a super commutative ring, right supermodules $M$ and $N$ can also be considered as left supermodules over $R$. In that case, $\phi(ax)=(-1)^{\vert\phi\vert\vert a\vert}a\phi(x)$ for all $x \in M$, $a \in R$ and $\phi$ homogeneous.

\begin{lemma}
\label{lemma:2.9}
Let $R$ be a super commutative ring, and $M$, $N$ supermodules over $R$. Then, $\mathrm{Hom}_R(M,N)$ is a supermodule over $R$.
\end{lemma}
\begin{proof}
Let $\phi:M\to N$ be a $R$-linear morphism.  Consider endomorphisms of $M$,  $\mathrm{P}^M_0,\mathrm{P}^M_1:M\to M$ given by $\mathrm{P}^M_0(m_0,m_1)=(m_0,0)$ and $\mathrm{P}^M_1(m_0,m_1)=(0,m_1)$.  Similarly we define $\mathrm{P}^N_\alpha:N\to N$, for $\alpha=0,1$. 

Let 
$$\phi_0=\mathrm{P}^N_0\circ \phi\circ \mathrm{P}^M_0+ \mathrm{P}^N_1\circ \phi\circ \mathrm{P}^M_1:M\to N$$ and $$\phi_1=\mathrm{P}^N_0\circ \phi\circ \mathrm{P}^M_1 + \mathrm{P}^N_1\circ \phi\circ \mathrm{P}^M_0:M\to N.$$  
By construction, it is clear that $\phi=\phi_0\oplus \phi_1$, and $\phi_0$ (respectively, $\phi_1$) is even (respectively, odd) morphism. This induces a natural grading on $\mathrm{Hom}_R(M,N)$ of $R$-linear morphism,
$$\mathrm{Hom}_R(M,N)=\bigl(\mathrm{Hom}_R(M,N)\bigr)_0\oplus \bigl(\mathrm{Hom}_R(M,N)\bigr)_1.$$ 
This makes $\mathrm{Hom}_R(M,N)$ a supermodule over $R$.
\end{proof}

Let $R$ be a super commutative ring.  We denote by $R$-$\mathrm{SMod}$  the category of supermodules over $R$, that is, $\mathrm{Ob}(R\text{-}\mathrm{SMod})$  is the collection of supermodules over $R$. And $\mathrm{Hom}_{R\text{-}\mathrm{SMod}}(M,N)$  is the set of $R$-linear morphisms from $M$ to $N$.  

\subsection{Free supermodules}
\label{subsection:2.3}
In this subsection we define \emph{free supermodules}, then state and prove the \emph{universal property of free supermodules} in superalgebra which is analogous to the universal property of free modules over a commutative algebra. 

\begin{definition}
\label{def:2.10}
A supermodule over super commutative ring $R$ is said to be \emph{free} if it has a basis formed by homogeneous elements.
Let $I$ and $J$ are arbitrary sets.  A supermodule $F$ over $R$ is said to be \emph{free of the type $(I,J)$} if there exists a set $\{f_i^0\in F_0\vert i\in I\}$ and $\{f_j^1\in F_1\vert j\in J\}$ such that 
$$F\cong \bigl(\oplus_{i\in I}  f_i^0 R \bigl)\oplus \bigl(\oplus _{j\in J}  f_j^1 R\bigr) .$$
\end{definition}

The definition for free supermodules is given in \cite{Rog2007} and \cite{Man1984} for the finite rank, where $I$ and $J$ are finite sets. But here we consider $I,J$ both to be arbitrary. If $f\in F$, then $f=\sum_{i\in I}f_i^0 c_i+\sum_{j\in J} f_j^1 c_j$ where $c_i,c_j\in R$ (may not be homogeneous), and $c_i=0$ (respectively, $c_j=0$) for all but finitely, many $i\in I$ (respectively, $j\in J$). 

\begin{notation}
\label{nota:2.11} Since $F$ is also a supermodule it is useful to write elements of $F$ in terms of $(x,y)$, where $x\in F_0$ even, and  $y\in F_1$ odd. Let $c_i=a_i+b_i$, where $a_i\in R_0$, and $b_i\in R_1$.  Similarly, $c_j=a_j + b_j$, where $a_j\in R_0$, and $b_j\in R_1$.  By the above discussion we get a general element of $F$ is of the form $(x,y)\in F_0\oplus F_1$, where $\displaystyle{x=\sum_{i\in I} f_i^0a_i + \sum_{j\in J} f_j^1} b_j $, and $\displaystyle{y= \sum_{i\in I} f_i^0 b_i + \sum_{j\in J} f_j^1  a_j}$, where $a_i, a_j \in R_0$, and $b_i, b_j \in R_1$, and $a_i=0$ and $b_i = 0$ (respectively, $a_j=0$ and $b_j = 0$) for all but finitely many $i\in I$ (respectively, $j\in J$).
\end{notation}

\begin{theorem}[\bfseries The universal property of free supermodules]\label{firstlabel}%
\label{theorem:2.12}	
Let $R$ be a super commutative ring, and $I,J$ arbitrary sets.  Let $F$ be a free supermodule over $R$ of the type $(I,J)$, and let $\mathcal{B}=\{f_i^0, f_j^1\vert i\in I, j\in J\}$ be a basis of $F$.  Let $M$ be a supermodule over $R$, and $\phi:\mathcal{B}\to M$ a set map. Then, there exists an unique $R$-linear morphism $\widetilde{\phi}:F\to M$ such that the following diagram commutes.
$$\begin{tikzcd}
\mathcal{B} \arrow[rd, "\phi"'] \arrow[rr, "i", hook] &   & F \arrow[ld, "!\exists\widetilde{\phi}"] \\ & M &                                         
\end{tikzcd}$$
\end{theorem}
\begin{proof}
Let  $f\in F$, then by definition $f=\sum_{i\in I}f_i^0 c_i+\sum_{j\in J} f_j^1 c_j$ where $c_i,c_j\in R$ (may not be homogeneous), and $c_i=0$ (respectively, $c_j=0$) for all but finitely many $i\in I$ (respectively, $j\in J$). Define,
\begin{equation}
\label{eq:1}
\widetilde{\phi}(f)=\widetilde{\phi}\left( \sum_{i\in I}f_i^0 c_i+\sum_{j\in J} f_j^1 c_j\right)=\sum_{i\in I} \phi(f_i^0)c_i+\sum_{j\in J} \phi(f_j^1) c_j.
\end{equation}
(Note that $\phi(f_i^0)$ and $\phi(f_\alpha^1)$ may not be homogeneous elements.)
Then, $\widetilde{\phi}$ is well defined and $R$-linear morphism. By definition, we get $\widetilde{\phi} \circ i=\phi$.
\end{proof}

When $F$ is considered as a left module over $R$, the degree needs to be considered while defining $\widetilde{\phi}$. If $\phi$ is an \emph{even set map}, that is, $\vert\phi(f)\vert=\vert f\vert$ for all homogeneous element $f\in \mathcal{B}$, then as a left module  $\widetilde{\phi}$ is simply $R$-linear morphism same as that of above theorem.  We call $\phi$ an \emph{odd set map} if $\vert\phi(f)\vert=\vert f\vert+1$.  In the following remark we describe $\widetilde{\phi}$ when $F$ considered as a left module over $R$.

\begin{remark} 
\label{remark:2.13} 
Notations as in Notation \ref{nota:2.11} and Theorem \ref{theorem:2.12}. Consider $F$ as a left supermodule over $R$, and $\phi$ an odd set map. Let $(x,y)\in F$ be general form
\[
x=\sum_{i\in I} a_if_i^0 + \sum_{j\in J} b_j f_j^1,\,\,
y=\sum_{i\in I} b_i f_i^0 + \sum_{j\in J} a_j f_j^1,\]
where $a_i, a_j \in R_0$, and $b_i, b_j \in R_1$, and $a_i=0$ and $b_i = 0$ (respectively, $a_j=0$ and $b_j = 0$) for all but finitely many $i\in I$ (respectively, $j\in J$). Let $\phi(f_i^0)=m_i^1\in M_1$, and $\phi(f_j^1) = m_j^0\in M_0$. Let $(x,y) \in F$ be defined as above. Then, 
\begin{align*}
\widetilde{\phi}(x,y)&=\sum_{i\in I} \widetilde{\phi}(a_if_i^0)+\sum_{j\in J}\widetilde{\phi}(b_j f_j^1) + \sum_{i\in I} \widetilde{\phi}(b_if_i^0)+\sum_{j\in J}\widetilde{\phi}(a_j f_j^1) \\
&=\sum_{i\in I} \widetilde{\phi}(f_i^0a_i)-\sum_{j\in J}\widetilde{\phi}(f_j^1b_j) + \sum_{i\in I} \widetilde{\phi}(f_i^0b_i)+\sum_{j\in J}\widetilde{\phi}( f_j^1 a_j) \\
&=\sum_{i\in I} m_i^1 a_i-\sum_{j\in J} m^0_j b_j + \sum_{i\in I} m_i^1b_i+\sum_{j\in J} m_j^0 a_j \\
&=\sum_{i\in I} a_i m_i^1 - \sum_{j\in J} b_j m_j^0 -\sum_{i\in I} b_i m_i^1 + \sum_{j\in J} a_j m_j^0.
\end{align*}	
\end{remark}

\subsection{Supermodules}
\label{subsection:2.2}
Superrings are non-commutative; hence we need to define both the left and right supermodules. However, if the ring is super commutative then by simple sign change, we can convert the left module into right and vice versa. This makes the category of supermodules over super commutative rings comparatively simple.  In this subsection, we will define the category of supermodules over a super commutative ring. 

\begin{definition}
\label{def:2.7}
Let $R$ be a superring, and $M$ a right (respectively, left) module over $R$. A module $M$ is a \emph{right} (respectively, \emph{left}) \emph{supermodule over $R$} if it has two submodules $M_0$ and $M_1$ such that $M=M_0\oplus M_1$, and if for all $\alpha, \beta\in \mathbb{Z}_2$, one has $ M_\beta R_\alpha\subset M_{\beta+\alpha}$ (respectively, $R_\alpha M_\beta\subset M_{\alpha+\beta}$).
\end{definition}

Let $R$ be a super commutative ring and $M$ be any right supermodule over $R$.  Define $ax=(-1)^{\mid x \mid \mid a \mid} xa$ for all $x\in M$, $a\in R$ homogeneous elements.  This gives $M$ a left module structure. Hence, $M$ can be considered both sided supermodule over super commutative ring $R$. In this paper, we are considering all rings to be super commutative. Hence, by the above observation, we will use the term \emph{supermodule over $R$} without mentioning right or left. Here onward we will consider modules as a right module over $R$, unless stated otherwise.

\begin{definition}
\label{def:2.8}
Let $R$ be a super commutative ring, and $M$, $N$ supermodules over $R$. We say that a morphism $\phi:M\to N$ is a \emph{$R$-linear morphism} if $\phi(xa) = \phi(x)a$ for all $x\in M$ and $a\in R$. We say that $\phi$ has \emph{degree} $\vert\phi\vert=\beta\in\mathbb{Z}_2$ if $\phi(M_\alpha)\subset N_{\alpha+\beta}$ for all $\alpha\in \mathbb{Z}_2$. If $\vert\phi\vert = 0$, then $\phi$ is said to be \emph{even morphism} and if $\vert\phi\vert = 1$, then $\phi$ is said to be \emph{odd morphism}.  If $\phi$ is even or odd, we say $\phi$ is \emph{homogeneous}.  
\end{definition}

Since $R$ is a super commutative ring, right supermodules $M$ and $N$ can also be considered as left supermodules over $R$. In that case, $\phi(ax)=(-1)^{\vert\phi\vert\vert a\vert}a\phi(x)$ for all $x \in M$, $a \in R$ and $\phi$ homogeneous.

\begin{lemma}
\label{lemma:2.9}
Let $R$ be a super commutative ring, and $M$, $N$ supermodules over $R$. Then, $\mathrm{Hom}_R(M,N)$ is a supermodule over $R$.
\end{lemma}
\begin{proof}
Let $\phi:M\to N$ be a $R$-linear morphism.  Consider endomorphisms of $M$,  $\mathrm{P}^M_0,\mathrm{P}^M_1:M\to M$ given by $\mathrm{P}^M_0(m_0,m_1)=(m_0,0)$ and $\mathrm{P}^M_1(m_0,m_1)=(0,m_1)$.  Similarly we define $\mathrm{P}^N_\alpha:N\to N$, for $\alpha=0,1$. 

Let 
$$\phi_0=\mathrm{P}^N_0\circ \phi\circ \mathrm{P}^M_0+ \mathrm{P}^N_1\circ \phi\circ \mathrm{P}^M_1:M\to N$$ and $$\phi_1=\mathrm{P}^N_0\circ \phi\circ \mathrm{P}^M_1 + \mathrm{P}^N_1\circ \phi\circ \mathrm{P}^M_0:M\to N.$$  
By construction, it is clear that $\phi=\phi_0\oplus \phi_1$, and $\phi_0$ (respectively, $\phi_1$) is even (respectively, odd) morphism. This induces a natural grading on $\mathrm{Hom}_R(M,N)$ of $R$-linear morphism,
$$\mathrm{Hom}_R(M,N)=\bigl(\mathrm{Hom}_R(M,N)\bigr)_0\oplus \bigl(\mathrm{Hom}_R(M,N)\bigr)_1.$$ 
This makes $\mathrm{Hom}_R(M,N)$ a supermodule over $R$.
\end{proof}

Let $R$ be a super commutative ring.  We denote by $R$-$\mathrm{SMod}$  the category of supermodules over $R$, that is, $\mathrm{Ob}(R\text{-}\mathrm{SMod})$  is the collection of supermodules over $R$. And $\mathrm{Hom}_{R\text{-}\mathrm{SMod}}(M,N)$  is the set of $R$-linear morphisms from $M$ to $N$.  

\subsection{Free supermodules}
\label{subsection:2.3}
In this subsection we define \emph{free supermodules}, then state and prove the \emph{universal property of free supermodules} in superalgebra which is analogous to the universal property of free modules over a commutative algebra. 

\begin{definition}
\label{def:2.10}
A supermodule over super commutative ring $R$ is said to be \emph{free} if it has a basis formed by homogeneous elements.
Let $I$ and $J$ are arbitrary sets.  A supermodule $F$ over $R$ is said to be \emph{free of the type $(I,J)$} if there exists a set $\{f_i^0\in F_0\vert i\in I\}$ and $\{f_j^1\in F_1\vert j\in J\}$ such that 
$$F\cong \bigl(\oplus_{i\in I}  f_i^0 R \bigl)\oplus \bigl(\oplus _{j\in J}  f_j^1 R\bigr) .$$
\end{definition}

The definition for free supermodules is given in \cite{Rog2007} and \cite{Man1984} for the finite rank, where $I$ and $J$ are finite sets. But here we consider $I,J$ both to be arbitrary. If $f\in F$, then $f=\sum_{i\in I}f_i^0 c_i+\sum_{j\in J} f_j^1 c_j$ where $c_i,c_j\in R$ (may not be homogeneous), and $c_i=0$ (respectively, $c_j=0$) for all but finitely, many $i\in I$ (respectively, $j\in J$). 

\begin{notation}
\label{nota:2.11} Since $F$ is also a supermodule it is useful to write elements of $F$ in terms of $(x,y)$, where $x\in F_0$ even, and  $y\in F_1$ odd. Let $c_i=a_i+b_i$, where $a_i\in R_0$, and $b_i\in R_1$.  Similarly, $c_j=a_j + b_j$, where $a_j\in R_0$, and $b_j\in R_1$.  By the above discussion we get a general element of $F$ is of the form $(x,y)\in F_0\oplus F_1$, where $\displaystyle{x=\sum_{i\in I} f_i^0a_i + \sum_{j\in J} f_j^1} b_j $, and $\displaystyle{y= \sum_{i\in I} f_i^0 b_i + \sum_{j\in J} f_j^1  a_j}$, where $a_i, a_j \in R_0$, and $b_i, b_j \in R_1$, and $a_i=0$ and $b_i = 0$ (respectively, $a_j=0$ and $b_j = 0$) for all but finitely many $i\in I$ (respectively, $j\in J$).
\end{notation}

\begin{theorem}[\bfseries The universal property of free supermodules]\label{firstlabel}%
\label{theorem:2.12}	
Let $R$ be a super commutative ring, and $I,J$ arbitrary sets.  Let $F$ be a free supermodule over $R$ of the type $(I,J)$, and let $\mathcal{B}=\{f_i^0, f_j^1\vert i\in I, j\in J\}$ be a basis of $F$.  Let $M$ be a supermodule over $R$, and $\phi:\mathcal{B}\to M$ a set map. Then, there exists an unique $R$-linear morphism $\widetilde{\phi}:F\to M$ such that the following diagram commutes.
$$\begin{tikzcd}
\mathcal{B} \arrow[rd, "\phi"'] \arrow[rr, "i", hook] &   & F \arrow[ld, "!\exists\widetilde{\phi}"] \\ & M &                                         
\end{tikzcd}
$$
\end{theorem}
\begin{proof}
Let  $f\in F$, then by definition $f=\sum_{i\in I}f_i^0 c_i+\sum_{j\in J} f_j^1 c_j$ where $c_i,c_j\in R$ (may not be homogeneous), and $c_i=0$ (respectively, $c_j=0$) for all but finitely many $i\in I$ (respectively, $j\in J$). Define,
\begin{equation}
\label{eq:1}
\widetilde{\phi}(f)=\widetilde{\phi}\left( \sum_{i\in I}f_i^0 c_i+\sum_{j\in J} f_j^1 c_j\right)=\sum_{i\in I} \phi(f_i^0)c_i+\sum_{j\in J} \phi(f_j^1) c_j.
\end{equation}
(Note that $\phi(f_i^0)$ and $\phi(f_\alpha^1)$ may not be homogeneous elements.)
Then, $\widetilde{\phi}$ is well defined and $R$-linear morphism. By definition, we get $\widetilde{\phi} \circ i=\phi$.
\end{proof}

When $F$ is considered as a left module over $R$, the degree needs to be considered while defining $\widetilde{\phi}$. If $\phi$ is an \emph{even set map}, that is, $\vert\phi(f)\vert=\vert f\vert$ for all homogeneous element $f\in \mathcal{B}$, then as a left module  $\widetilde{\phi}$ is simply $R$-linear morphism same as that of above theorem.  We call $\phi$ an \emph{odd set map} if $\vert\phi(f)\vert=\vert f\vert+1$.  In the following remark we describe $\widetilde{\phi}$ when $F$ considered as a left module over $R$.

\begin{remark} 
\label{remark:2.13} 
Notations as in Notation \ref{nota:2.11} and Theorem \ref{theorem:2.12}. Consider $F$ as a left supermodule over $R$, and $\phi$ an odd set map. Let $(x,y)\in F$ be general form
\[
x=\sum_{i\in I} a_if_i^0 + \sum_{j\in J} b_j f_j^1,\,\,
y=\sum_{i\in I} b_i f_i^0 + \sum_{j\in J} a_j f_j^1,\]
where $a_i, a_j \in R_0$, and $b_i, b_j \in R_1$, and $a_i=0$ and $b_i = 0$ (respectively, $a_j=0$ and $b_j = 0$) for all but finitely many $i\in I$ (respectively, $j\in J$). Let $\phi(f_i^0)=m_i^1\in M_1$, and $\phi(f_j^1) = m_j^0\in M_0$. Let $(x,y) \in F$ be defined as above. Then, 
\begin{align*}
\widetilde{\phi}(x,y)&=\sum_{i\in I} \widetilde{\phi}(a_if_i^0)+\sum_{j\in J}\widetilde{\phi}(b_j f_j^1) + \sum_{i\in I} \widetilde{\phi}(b_if_i^0)+\sum_{j\in J}\widetilde{\phi}(a_j f_j^1) \\
&=\sum_{i\in I} \widetilde{\phi}(f_i^0a_i)-\sum_{j\in J}\widetilde{\phi}(f_j^1b_j) + \sum_{i\in I} \widetilde{\phi}(f_i^0b_i)+\sum_{j\in J}\widetilde{\phi}( f_j^1 a_j) \\
&=\sum_{i\in I} m_i^1 a_i-\sum_{j\in J} m^0_j b_j + \sum_{i\in I} m_i^1b_i+\sum_{j\in J} m_j^0 a_j \\
&=\sum_{i\in I} a_i m_i^1 - \sum_{j\in J} b_j m_j^0 -\sum_{i\in I} b_i m_i^1 + \sum_{j\in J} a_j m_j^0.
\end{align*}	
\end{remark}

\section{Super projective modules}
\label{section:03}
The class of projective modules expands the class of free modules over a ring, by holding onto some of the main properties of free modules. 
In this section, we define super projective modules analogous to projective modules in ordinary algebra and discuss some of their properties and examples.

\begin{theorem}
\label{theorem:3.1}
Let $R$ be a super commutative ring.  Let $P$ be a supermodule over $R$. Then the following are equivalent:
\begin{enumerate}
\item A supermodule $P$ over $R$ is a direct summand of free supermodule, that is, there exists supermodule $Q$ over $R$ such that $P\oplus Q$ is isomorphic to $F$, where $F$ is of type $(I, J)$ for some $I, J$ arbitrary sets.
\item For any $R$-homomorphism $h:P \to N$, and a surjective morphism $g\in \mathrm{Hom}_{R\text{-}\mathrm{SMod}}(M,N)$, there exists a unique $\widetilde{h}:P\to M$ a $R$-homomorphism such that $g\circ \widetilde{h}=h$, that is, the following diagram commutes.
$$\begin{tikzcd}
                  & P \arrow[ld, "!\exists \widetilde{h}"', dashed] \arrow[d, "h"] &   \\
M \arrow[r, "g"'] & N \arrow[r]                                                    & 0
\end{tikzcd}$$
\item If $\begin{tikzcd} 0 \arrow[r] & E_{1} \arrow [r,"h"] & E_{2} \arrow[r,"g"] & P \arrow[r] & 0 
\end{tikzcd}$ is a short exact sequence of supermodules over $R$, then it splits, that is, there exists $s:P\to E_2$ such that $g\circ s=\mathrm{Id}_P$. 
\end{enumerate}
\end{theorem}
\begin{proof} Let $\iota:P\hookrightarrow F$ be an inclusion map, and $\pi:F\to P$ a projection map. Hence, $\pi\circ \iota=\mathrm{Id}_P$, where $\mathrm{Id}_P$ is the identity map. Let $F= \bigl(\oplus_{i\in I} f_i^0 R \bigr) \oplus \bigl(\oplus_{j\in J} f_j^1 R \bigr)$. Let $h\circ \pi(f_i^0)=n_i \in N$.  Similarly, $h\circ \pi(f_j^1)=n_j\in N$. Since $g$ is surjective there exists $m_i \in M$ such that $g(m_i)=n_i$ and $m_j\in M$ such that $g(m_j)=n_j$. By the universal property of free supermodules over $R$ (Theorem \ref{theorem:2.12}), we get a morphism of supermodules $k:F\to M$, such that $k(f_i^0)=m_i$, and $k(f_j^1)=m_j$. Let $\widetilde{h}=k\circ \iota:P\to M$ be the required map. To prove the claim we need to check $g\circ \widetilde{h}=h$.  Let $p\in P$, be given by 
$p=\sum_{i\in I} f_i^0 a_i + \sum_{j\in J} f_j^1 a_j$,  $a_i$ and $a_j \in R$,  
$a_i=0$ (respectively, $a_j = 0$) for all but finitely many $i\in I$ (respectively, $j\in J$). Then, 
\begin{align*}
g\circ \widetilde{h}(p)&=g\circ k \circ \iota (p)
=g\circ k \left(\sum_{i\in I} f_i^0 a_i + \sum_{j\in J} f_j^1 a_j \right)\\
&=g\left(\sum_{i\in I}m_i  a_i + \sum_{j\in J} m_j a_j \right) 
=\sum_{i\in I} n_i a_i + \sum_{j\in J} n_j a_j.
\end{align*} 
On the other hand, 
\begin{align*}
h(p)&= h\left(\sum_{i\in I} f_i^0 a_i + \sum_{j\in J} f_j^1 a_j \right) 
= h\circ \pi\circ \iota \left(\sum_{i\in I} f_i^0 a_i + \sum_{j\in J} f_j^1 a_j \right)\\
&= h\left(\sum_{i\in I} \pi(f_i^0) a_i + \sum_{j\in J} \pi(f_j^1) a_j \right)
= \sum_{i\in I} h \circ \pi(f_i^0) a_i + \sum_{j\in J} h \circ \pi(f_j^1) a_j \\
&=\sum_{i\in I} n_i a_i + \sum_{j\in J} n_j a_j.
\end{align*}
Hence $g\circ \widetilde{h}=h$. Therefore (1) implies (2). 
Let us now prove (2) implies (3).  For this $\begin{tikzcd} 0 \arrow[r] & E_1\arrow[r, "h"] & E_2 \arrow[r, "g"] & P\arrow[r] & 0\end{tikzcd}$.  Let $\mathrm{Id}_P:P\to P$, then by $(2)$ there exists $s$  such that $g\circ  s=\mathrm{Id}_P$.  Hence, $s$ is a section, that is, the exact sequence splits.
To prove (3) implies (1),  let $I=P_0$ and $J=P_1$, where  $P=P_0\oplus P_1$.  Consider a free supermodule over $R$ of the type $(I,J)$, defined by $$F=\sum_{p^0\in P_0} p^0 R\oplus \sum_{p^1\in P_1} p^1 R.$$
Let $g: F\to P$ be a morphism such that $p^0\mapsto p^0$, and $p^1\mapsto p^1$, obtained by the universal property of free supermodules. This gives us the following exact sequence,  $$\begin{tikzcd}
0 \arrow[r] & \text{Ker } g \arrow[r, "\iota", hook] & F \arrow[r, "g"] & P \arrow[r] & 0
\end{tikzcd}.$$  By (3) we get a section $s:P\to F$, such that $g\circ s=\mathrm{Id}_P$, 
$$\begin{tikzcd}
0 \arrow[r] & \text{Ker } g \arrow[r, "\iota", hook] & F \arrow[r, "g", bend left] & P \arrow[l, "s", bend left] \arrow[r] & 0
\end{tikzcd}.$$  Using the above exact sequence we will show that $F\cong P\oplus \mathrm{Ker} g$.  Define $\phi:F\to P\oplus\mathrm{Ker} g$, $x\mapsto (g(x), x-s\circ g(x))$.  Then $\phi^{-1}:P\oplus \mathrm{Ker} g\to F$, is given by $(p,h)\mapsto s(p)+h$.  It is easy to check that $\phi$, and $\phi^{-1}$ are $R$-homomorphisms. 
\end{proof}

\begin{definition}
\label{definition:3.2}
A supermodule $P$ over $R$ is called \emph{super projective} if it satisfies any of the equivalent conditions of Theorem \ref{theorem:3.1}.
\end{definition}

\begin{example}
\label{example:3.3}
Let  $R = \mathbb{Z}_6[\xi_1, \xi_2]$ be a super commutative ring, where $\xi_1$ and $\xi_2$ are odd variables, that is, $\xi_1 \xi_2 = -\xi_2 \xi_1$ and $\xi_1^2 = \xi_2^2 = 0$. We know that $R$ is a supermodule over $R$. Let  $P = \mathbb{Z}_2[\xi_1, \xi_2]$ be a super submodule of $R$. We claim that $R$ is a super projective module over $R$ which is not free supermodule. It is easy to see that $P \oplus \mathbb{Z}_3[\xi_1, \xi_2] = R$. Therefore, $P$ is super projective module over $R$. But $P$ is not free supermodule since any free supermodule of $\mathbb{Z}_6[\xi_1, \xi_2]$	has cardinality as power of $6^4$ whereas cardinality of $P$ is $2^4$.
\end{example}

We can construct new super projective modules using direct sums and tensor products. The following proposition gives a way to construct new projective modules using the old ones. 

\begin{proposition}
\label{proposition:3.4}
Let $P$, $P_1$ and $P_2$ be super projective modules over $R$. Then,
\begin{enumerate}
\item the tensor product of $P_1$ and $P_2$ is super projective module over $R$.
\item the direct sum of $P_1$ and $P_2$ is super projective module over $R$.
\item the module of endomorphisms of $P$ over $R$, $\mathrm{End}_R (P)$ is a super projective module over $R$.
\end{enumerate}
\end{proposition}
\begin{proof} 
Since $P_1$ and $P_2$ are super projective modules over $R$, there exist $Q_1$, $Q_2$ supermodules such that $P_1 \oplus Q_1 = R^{(I_1, J_1)}$ and $P_2 \oplus Q_2 = R^{(I_2, J_2)}$, for some indexing sets, $I_1$, $I_2$, $J_1$ and $J_2$.

We will first prove that the tensor product of $P_1$ and $P_2$ is a super projective module over $R$.  Consider the tensor product of free supermodules, $ R^{(I_1, J_1)}\otimes  R^{(I_2, J_2)}= (P_1 \oplus Q_1)\otimes (P_2\oplus Q_2)$. Since the tensor product and direct sum commutes, we get 
\begin{align*}
(P_1 \otimes P_2) \oplus \bigr((P_1 \otimes Q_2) \oplus (Q_1 \otimes P_2) \oplus (Q_1 \otimes Q_2)\bigl) &= R^{(I_1, J_1)} \otimes R^{(I_2, J_2)} \\
&= R^{\bigl( (I_1 I_2 + J_1 J_2), (I_1 J_2 + J_1 I_2) \bigl)}.
\end{align*}
Therefore, $P_1 \otimes P_2$ is a super projective module over $R$.  

Next, we prove that the direct sum of $P_1$ and $P_2$ is a super projective module over $R$. Let $P_1$ and $P_2$ be as defined above. Then, $P_1 \oplus Q_1 = R^{(I_1, J_1)}$ and $P_2 \oplus Q_2 = R^{(I_2, J_2)}$. By taking the direct sum of these two we get, 
$$\bigr( (P_1 \oplus P_2) \oplus (Q_1 \oplus Q_2) \bigl) =  R^{(I_1, J_1)} \oplus R^{(I_2, J_2)} =  R^{(I_1 + I_2, J_1 + J_2)}.$$
Therefore, $P_1 \oplus P_2$ is a super projective module over $R$.

Let $P\oplus Q=F$, where $F$ is a free supermodule, and $Q$ a supermodule.  We know that $\mathrm{End}_R (F) = \mathrm{End}_R (P) \oplus \mathrm{End}_R (Q)$. Since $\mathrm{End}_R(F)$ is free, we get $\mathrm{End}_R (P)$ is a super projective module. \end{proof}

\section{Superspheres and super projective modules}
\label{section:04}
In this section we give two classes of super projective modules. The first class is obtained from the global sections of tangent bundle over a supersphere, that is, the module of vector fields over a supersphere. The other class of super projective modules is obtained from projectors over the supersphere.  This example is motivated by physics. The supersphere is a \emph{supermanifold}. Since there are many kinds of supermanifolds, we will give the definition of supermanifold which we will be using in this paper.

\subsection{Supermanifolds}
\label{subsection:supermanifold}
To find the super analog for this example, we need to study the supersphere and the tangent bundle over it. Since the supersphere is a supermanifold, we recall certain key results for supermanifolds in general. There are two main ways to deal with these and the one which we will be considering is the \emph{concrete approach}, where a supermanifold is a topological space, with charts to a \textit{real superspace} that are \textit{supersmooth}. This approach was developed greatly by Rogers \cite[pp.~51]{Rog2007}. The other approach treats supermanifolds as ringed spaces with a sheaf of \textit{supersmooth functions} (refer to \cite{BBH1991}). To make this article self-contained, we are giving definitions of $G^\infty$ DeWitt supermanifold. But for details, we suggest readers refer to the above-cited books, namely \cite{Rog2007} and \cite{BBH1991}.

The \emph{real superspace} is defined as 
$$\mathbb{R}_{S}^{m,n}=\underbrace{\mathbb{R}_{S_{0}} \times \cdots \times \mathbb{R}_{S_{0}}}_{m \text { copies }} \times \underbrace{\mathbb{R}_{S_{1}} \times \cdots \times \mathbb{R}_{S_{1}}}_{n \text { copies }},$$
where $\mathbb{R}_S$, $\mathbb{R}_{S_{0}}$ and $\mathbb{R}_{S_{1}}$ are as defined in Example \ref{example:2.5}.
An element in $\mathbb{R}_{S}^{m,n}$ is typically written as $(x,\xi)= (x_1,x_2,\dots,x_m,\xi_{1},\xi_{2},\dots,\xi_{n})$. One can check $\mathbb{R}_{S}^{m,n}$ is super commutative.  There is a canonical map $\epsilon:\mathbb{R}_S\to\mathbb{R}_S$  defined on the generators of $\mathbb{R}_S$ as $\epsilon(\beta_{[i]})=0$ and 
$\epsilon(1)=1$. This is called the \emph{body map} \cite[pp.~19]{Rog2007}.  Let $\sigma: \mathbb{R}_{S}\rightarrow \mathbb{R}_{S}$ be defined as $\sigma(x)=x-\epsilon(x)$. This is known as the \emph{soul map} \cite[pp.~19]{Rog2007}.
The body map can be extended to the real space as follows:
The \emph{$(m,n)$ body map} is defined as $\epsilon^{m,n}: \mathbb{R}_{S}^{m,n}\rightarrow \mathbb{R}^{m}$ with $$\epsilon^{m,n}(x,\xi)= \bigl(\epsilon(x_{1}),\epsilon(x_{2}),\dots,\epsilon(x_{m})\bigr)$$

Consider $U\subset \mathbb{R}_{S}^{m,n}$. This set is open in \emph{DeWitt topology} if and only if there exists an open set $V\subset \mathbb{R}^{m}$ such that $(\epsilon^{m,n})^{-1}(V)= U$.  

As in traditional manifolds, there needs to be a concept analogous to smooth functions on $\mathbb{R}^{m}$. These are named \emph{supersmooth functions} on $\mathbb{R}_{S}^{m,n}$. There are several classes of such functions, which are described in \cite[Chapter 4]{Rog2007}. Most of these functions are defined as a series of smooth functions which are continued analytically from $\mathbb{R}^{m}$ to $\mathbb{R}^{m,0}$ and then to the $(m,n)$-superspace. Recall that, if $V\subset \mathbb{R}^{m}$ then $f:V\rightarrow E$, where $E$ is a vector space over $\mathbb{R}$, is smooth if $f=\sum f_{i}e_{i} \in C^{\infty}(V,E)$, where $\{e_{i}\}$ is some  basis of $E$ and $f_{i}:V\rightarrow \mathbb{R}$ are smooth.

\begin{definition}
\label{def:4.1}
Suppose $V \subset \mathbb{R}^{m}, U \subset \mathbb{R}_{S}^{m,n}$ such that $\epsilon^{m,n }(U)=V$.\\
Define $$\widehat{f}(x, \xi):=\sum_{M_{[\infty]}} \frac{1}{i_{1}! \ldots i_{m} !} \partial_{1}^{i_{1}} \cdots \partial_{m}^{i_{m}} f\left(\epsilon^{m,n}(x)\right) \sigma\left(x_{1}\right)^{i_{1}} \ldots \sigma\left(x_{m}\right)^{i_{m}}$$
where $f \in C^{\infty}\left(V, \mathbb{R}_{S}\right)$. We call $\widehat{f}$ the \emph{Grassmann analytic continuation of $f$} \cite[Definition 4.2.2, pp.~36]{Rog2007}.
\end{definition}

Using the above definition, we say, given $U$ open in $\mathbb{R}_{S}^{m,n}$, a function $f: U \rightarrow \mathbb{R}_{S}$ is said to be \emph{$G^{\infty}$ on $U$} if and only if there exists a collection $\left\{f_{\underline{\mu}} \mid \underline{\mu} \in M_{[n]}\right\}$ of $C^{\infty}\left(\epsilon^{m,n}(U),\mathbb{R}_{S}\right)$ functions such that
$$f(x, \xi)=\sum_{\underline{\mu} \in M_{[n]}} \widehat{f_{\underline{\mu}}}(x) \xi_{\underline{\mu}}$$
for each $(x, \xi)$ in $U$. Here $\underline{\mu}$ is a multi index, and $M_{[n]}$ is the set of multi indices as defined in Notation \ref{nota:2.3}, $\{\underline{\mu}=\mu_{1}\mu_{2}\dots\mu_{k}\mid 1\leq \mu_{1}< \mu_{2}< \cdots < \mu_{k}\leq n\}$, $\xi_{\underline{\mu}}=\xi_{\mu_{1}}\xi_{\mu_{2}}\dots \xi_{\mu_{k}}$. This expansion is called the \emph{Grassmann analytic expansion of $f$}, and the functions $f_{\underline{\mu}}$ are called the \emph{Grassmann analytic coefficients of $f$} \cite[Definition 4.3.2, pp.~39]{Rog2007}.
Finally, we define a `concrete' supermanifold.
\begin{definition}[$(\mathbb{R}_{S}^{m, n},DeWitt,G^{\infty})$-Supermanifold]
Let $X$ be a set, and $\mathbb{R}_{S}^{m, n}$ be equipped with the DeWitt topology. For $U\subset X$, define charts to be the pair $(U, \phi_{U})$ where $\phi_{U} : U \rightarrow \widehat{U}$ is a bijective map for some open $\widehat{U}\subset\mathbb{R}_{S}^{m,n}$. Suppose there exists an atlas, i.e.  $\mathcal{A} := \{ \phi_{U_{\alpha}}\lvert U_{\alpha}\subset X, \bigcup_{\alpha} U_{\alpha}=X \}$ and suppose $\phi_{U} \circ \phi_{V}^{-1}$ is $G^{\infty}$ on the intersection of any two charts in the atlas, then $X$ is a $(\mathbb{R}_{S}^{m, n},DeWitt,G^{\infty})$-supermanifold. 
\end{definition}

\begin{remark}
\label{remark:4.2} We can also define a supermanifold where topological space is locally homeomorphic to $\mathbb{R}_{S[L]}^{m,n} = \mathbb{R}_{S[L,0]}^{m}\oplus \mathbb{R}_{S[L,1]}^{n}$. We can define DeWitt topology and $G^\infty$ function similar to Definition \ref{def:4.1}.
\end{remark}

\begin{proposition}
\label{proposition:4.3} The \emph{supercircle} $S^{1,0}=\{(x,y) \in \mathbb{R}_{S}^{2,0} \vert  x^2 + y^2 = 1\}$  is a $(1,0)$-dimensional supermanifold.
\end{proposition}
\begin{proof} consider a map $\phi: \mathbb{R}_{S}^{2,0} \longrightarrow \mathbb{R}_{S}$ given by $\phi (x , y) = x^2 + y^2$, where
\begin{equation}
\label{eq:20}
x= \sum_{\underline{\lambda}=\lambda_{1}\lambda_{2}\dots \lambda_{k} \in M_{0}} x_{\underline{\lambda}}\beta_{[\lambda_{1}]}\beta_{[\lambda_{2}]}\dots\beta_{[\lambda_{k}]},
\, y=\sum_{\underline{\lambda}=\lambda_{1}\lambda_{2}\dots \lambda_{k} \in M_{0}} y_{\underline{\lambda}}\beta_{[\lambda_{1}]}\beta_{[\lambda_{2}]}\dots\beta_{[\lambda_{k}]},
\end{equation}
where $x_{\underline{\lambda}},\, y_{\underline{\lambda}}\in \mathbb{R}$. Then, $S^{1,0}= \phi^{-1}(1)$, that is, $S^{1,0}$ is a level set.
The equation $x^2 + y^2 = 1$ of the supercircle gives relation between $x_{\underline{\lambda}}$ and $y_{\underline{\lambda}}$. For this let us compare the constant term of the equation.  We get the first equation as 
$x_{\underline{0}} ^2 + y_{\underline{0}}^2 = 1.$ 
By induction we can see that for a multi index $\underline{\lambda}=\lambda_1\lambda_2 \ldots \lambda_k$,
\begin{equation} 
\label{eq:21}
x_{\underline{0}} x_{\underline{\lambda}} + y_{\underline{0}} y_{\underline{\lambda}} + \Sigma_{\underline{\mu},\underline{\nu}} \delta_{\underline{\mu}\underline{\nu}}(x_{\underline{\mu}} x_{\underline{\nu}} +  y_{\underline{\mu}} y_{\underline{\nu}})= 0,
\end{equation}
where, $\underline{\mu}=\mu_{1}\mu_{2}\dots \mu_{r}$, $\underline{\nu}=\nu_{1}\nu_{2}\dots \nu_{s}$, such that $\{\mu_1,\ldots,\mu_r,\nu_1,\ldots,\nu_s\}=\{\lambda_1,\ldots,\lambda_k\}$ and $\{\mu_1,\ldots,\mu_r\}\cap\{\nu_1,\ldots,\nu_s\}=\emptyset$, and $\delta{\underline{\mu}\underline{\nu}}$ is $1$ or $0$ depending upon the sign of $\underline{\mu}\underline{\nu}$.
Now we are ready to prove $S^{1,0}$ is a supermanifold.  For this, we need to find an atlas. Consider the map
\begin{align*}
\label{eq:22}
\psi_{x^+}: U_{x^+} = \{(x,y) \in S^{1,0} \vert \epsilon(x) = x_{\underline{0}} > 0\} &\longrightarrow (\epsilon^{1,0})^{-1}((-1,1))\\
(x,y) &\longmapsto x,
\end{align*}
where $(-1,1)$ is the open interval in $\mathbb{R}$. The inverse map $\psi ^{-1}_{x^+}:(\epsilon^{1,0})^{-1}((-1,1))\to U_{x^+}$ is $y\mapsto (\sqrt{1-y^2} , y)$. Here $\sqrt{1-y^2}$ is defined as follow: Let $y$ and $\sqrt{1-y^2}=x$ be as in \eqref{eq:20}. We need to find $x_{\underline{\lambda}}$ for all multi index $\underline{\lambda}$.  We will do this by induction on the length of $\underline{\lambda}$.  When $l(\underline{\lambda})=0$, then $x_{\underline{0}}=\sqrt{1-y_{\underline{0}}^2}$. Note that, since $y_{\underline{0}}\neq \pm 1$, $x_{\underline{0}}\neq 0$. By the induction hypothesis, let us assume that $x_\mu$ is known for $l(\mu)< k$.  By \eqref{eq:21} 
$$x_{\underline{\lambda}}=\frac{-1}{x_{\underline{0}}}(y_{\underline{0}} y_{\underline{\lambda}} + \Sigma_{\underline{\mu},\underline{\nu}} \delta_{\underline{\mu}\underline{\nu}}(x_{\underline{\mu}} x_{\underline{\nu}} +  y_{\underline{\mu}} y_{\underline{\nu}})).$$
Hence by the mathematical induction, $\sqrt{1-y^2}$ is defined. Note that since $y\in \mathbb{R}_{S_{0}}$, so is $x$. By construction $\psi_{x^+}$ is $G^\infty$.  Hence $(U_{x^+},\psi_{x^+})$ is a chart. Now consider the inverse of $\psi_{x^-}$ is given by $y\mapsto(-\sqrt{1-y^2},y)$, where $\sqrt{1-y^2}$ is as defined above. Hence, we get a chart $(U_{x^-},\psi_{x^-})$.  Similarly we define charts $(U_{y^+},\psi_{y^+})$ and $(U_{y^-},\psi_{y^-})$.
These four charts will give us an atlas, which shows that $S^{1,0}$ is a supermanifold of dimension $(1,0)$. 
\end{proof}

\begin{proposition}
\label{proposition:4.4}
Let $S^{1,0}$ be the supercirlce. Then the tangent bundle over $S^{1,0}$ is trivial.
\end{proposition}
\begin{proof}
We will show that the tangent bundle of the supercircle is given by
$$\{ (x,y,-\lambda y,\lambda x)\mid \lambda \in \mathbb{R}_{S}, (x,y)\in S^{1,0} \}.$$
Define $\overline{\sin}:\mathbb{R}_{S_{0}}\rightarrow \mathbb{R}_{S}$ as the formal sum
$$\overline{\sin} \theta = \sum_{i=0}^{\infty} (-1)^{i}\frac{\theta^{2i+1}}{(2i+1)!}.$$
Similarly, define $\overline{\cos}:\mathbb{R}_{S_{0}}\rightarrow \mathbb{R}_{S}$ as the formal sum
$$\overline{\cos} \theta = \sum_{i=0}^{\infty} (-1)^{i}\frac{\theta^{2i}}{(2i)!}.$$
By properties of even elements, the images of both these functions turn out to be in $\mathbb{R}_{S_{0}}$ and hence is a valid candidate for parametrizing the supercircle.
Using the superderivation given by the supersmooth structure of $S^{1,0}$, we get
\begin{equation}
\label{eqn:sin}
D\overline{\sin} \theta = \partial_{E} \sum_{i=0}^{\infty} (-1)^{i}\frac{\theta^{2i+1}}{(2i+1)!}
= \sum_{i=0}^{\infty} \partial_{E} (-1)^{i}\frac{\theta^{2i+1}}{(2i+1)!}
= \sum_{i=0}^{\infty} (-1)^{i}\frac{\theta^{2i}}{(2i)!}
= \overline{\cos} \theta 
\end{equation}
Similarly, 
\begin{equation}
\label{eqn:cos}
D\overline{\cos} \theta = -\overline{\sin} \theta
\end{equation} 
Now, since $D$ is a superderivation, and hence a derivation on the underlying manifold, for any function $f:\mathbb{R}_{S_{0}}\rightarrow \mathbb{R}_{S}$ we get $f$ is a constant if $Df = 0$. Using the previous result (\ref{eqn:sin}) and (\ref{eqn:cos}), we have
\begin{align*}
D(\overline{\sin^{2}}\theta + \overline{\cos^{2}}\theta) &= 2\overline{\sin \theta}D\overline{\sin \theta} + 2\overline{\cos \theta}D\overline{\cos \theta} = 0
\end{align*}
Since $\overline{\sin} 0 = 0$ and $\overline{\cos} 0 = 1$, we have $\overline{\sin^{2}}\theta + \overline{\cos^{2}}\theta = 1$
Thus, $(\overline{\cos} \theta,\overline{\sin} \theta)$ parametrizes the supercircle $S^{1,0}= \{ (x,y)\in \mathbb{R}_{S}^{2,0}\vert x^{2}+y^{2}=1 \}$.
Hence, at a point $(\overline{\cos} \theta_{0},\overline{\sin} \theta_{0})$, the tangent vector will be $(-\lambda\overline{\sin} \theta_{0},\lambda\overline{\cos} \theta_{0})$ where $\lambda\in \mathbb{R}_{S}$. Thus, it follows that the tangent bundle of the supercircle is the trivial bundle $S^{1,0}\times \mathbb{R}_{S}$.\end{proof}
In general, it is possible that the tangent bundle of a super Lie group is trivial, which might be something we prove in future works.

\begin{remark}
\label{remark:4.5}
We can define supersphere of even dimension as $S^{m,0}=\{(x_0,\ldots,x_m)\in\mathbb{R}_{S}^{m+1,0}\mid  x_0^2+\cdots+x^m=1\}$.  We can also replace $\mathbb{R}_S^{m,0}$ with $\mathbb{R}_{S[L]}^{m,0}$ as follows:
$S^{m,0}=\{(x_0,\ldots,x_m)\in\mathbb{R}_{S[L]}^{m+1,0}\mid x_0^2+\cdots+x_m^0=1\}$.  The above proof can be generalized with little more calculation to show that in both cases the supersphere $S^{m,0}$ is a supermanifold of even dimension $(m,0)$.	
\end{remark}

\subsection{Super vector fields}
\label{subsec:4.2}
Kaplansky constructed projective modules using tangent bundles of spheres, \cite[Example 1, pp.~269]{Swa1962}.  Let us recall the construction.  Let $S^n=\{(x_o,x_1,\ldots,x_n)\mid \sum_{i=0}^n x_i^2=1\}$ be the unit sphere, and $C(S^n)$ the ring of continuous real-valued functions over $S^n$.  Consider the quotient ring $\Lambda= \mathbb{R}[x_0,x_1,\ldots,x_n]/(x_0^2 + x_1^2 + \cdots +x_n^2-1)$. Then $\Lambda$ can be identified with the subring of $C(S^n)$. Let $F$ be a free module over $\Lambda$ on $(n+1)$ generators $s_0,s_1,\ldots , s_n$. Define $g:F \to F$ by $g(s_i) = \displaystyle{\Sigma _j x_i x_j s_j}$. Then $g$ is idempotent, and $\ker g$ is a projective module over $\Lambda$.  Furthermore $C(S^n)\otimes P\cong \Gamma(S^n,TS^n)$, where $\Gamma(S^n,TS^n)$ is the module of global sections of the tangent bundle. Since we know that $TS^n$ is not trivial for $n\neq 0,1,3,7$, by The Swan's theorem, \cite[Theorem 2, pp.~268]{Swa1962}$, C(S^n)\otimes P$ is not free for $n\neq 0,1,3,7$.  Moreover, $\Img g$ is free, so $P$ is indeed a stably free module. This example yields many ``non-trivial" projective modules.

\begin{proposition}
\label{proposition:4.6}
Let $F$ be a free supermodule over a superring $R$. Define $g:F \longrightarrow F$, where $g$ is idempotent, that is, $g^2 = i$. Now consider the sequence, 
$$\begin{tikzcd}
0 \arrow[r] & \ker g \arrow[r, "\iota", hook] & F \arrow[r, "g"] & \mathrm{Im } \: g \arrow[r]                         & 0  
\end{tikzcd}$$
Then, $\text{ker } g \oplus \text{Im } g \cong F$. 
\end{proposition}
\begin{proof}
Let $x \in F$ then $g(x) \in F$ lets say $g(x) = y$. Now, $g(y) = g(g(x)) = g^2(x) = g(x)$. Consider $x - g(x)$, 
$$g(x - g(x)) = g(x) - g^2(x) = g(x) - g(x) = 0.$$
Therefore, $x - g(x) \in \ker g$. But $x = g(x) + (x - g(x))$. Hence $F \cong \Ima g + \ker g$. Now we want to prove that it is a direct sum. Suppose, $x \in \ker g \cap \Ima g$. Let $x = g(y)$ as $x \in \Ima g$. Now, $g(x) = g(y)$ as $g(x) = g^2(y) = g(y)$. So, $g(x) \in \Ima g$, but $g(x) = 0$ because $x \in \ker g$. hence, we get $x = g(y) = g(x) = 0$. Therefore, $x = 0$. Hence, $\ker g + \Ima g$ is direct sum and $F \cong \ker g \oplus \Ima g$.
\end{proof}

\begin{example}
\label{example:4.7}
Let $R$ be a super commutative ring and consider $\Lambda = R [x_0, x_1, \ldots , x_n] / (x_0^2 + x_1^2 + \cdots + x_n^2 - 1)$, $x_i$ are commutative.  We denote by $\bar{x}_i$ the image of $x_i$ under the canonical quotient map from $R[x_0,\ldots,x_n]$ to $\Lambda$, for all $i=0,\ldots,n$. In particular $\sum_{i=0}^n \bar{x}_i^2=1$. Also, $\Lambda=R_0[\bar{x}_0,\ldots,\bar{x}_n]\oplus R_1[\bar{x}_0,\ldots,\bar{x}_n]$, so is clearly a super commutative ring. 
Let $F$ be a free super $\Lambda$-module on $n$ generators $s_0, s_1, \ldots ,s_n$. Define $g:F\to F$ by $g(s_i) = \Sigma _{j=0}^n \bar{x}_i \bar{x}_j s_j$ for all  $i= 0,1, \ldots ,n$. We get, 
\begin{align*}
g^2(s_i) &= g(g(s_i)) = g(\bar{x}_i\bar{x}_0 s_0 +\cdots + \bar{x}_i^2s_i + \cdots + \bar{x}_i \bar{x}_n s_n) \\
&= \bar{x}_i\bar{x}_0(\bar{x}_0^2 s_0 + \cdots + \bar{x}_0 \bar{x}_n s_n) +\cdots + \bar{x}_i^2 (\bar{x}_i \bar{x}_0 s_0 + \cdots + \bar{x}_i \bar{x}_n s_n) + \\
&\hspace*{6cm}\cdots + \bar{x}_i \bar{x}_n(\bar{x}_n \bar{x}_0 s_0 + \cdots + \bar{x}_n^2s_n) \\
&= \bar{x}_i\bar{x}_0^3 s_0 + \cdots + \bar{x}_i\bar{x}_0^2 \bar{x}_n s_n +\cdots + \bar{x}_i^3 \bar{x}_0 s_0 + \cdots + \bar{x}_i^3\bar{x}_n s_n + \\
&\hspace*{6cm} \cdots + \bar{x}_i \bar{x}_n^2\bar{x}_0 s_0 + \cdots + \bar{x}_i \bar{x}_n^3 s_n\\
&= \bar{x}_i\bar{x}_0 (\bar{x}_0^2+\cdots+\bar{x}_i^2+\cdots + \bar{x}_n^2) s_0 + \cdots + \bar{x}_i \bar{x}_n (\bar{x}_0^2+ \cdots +\bar{x}_i^2+\cdots + \bar{x}_n^2) s_n   \\ 
&= \bar{x}_i\bar{x}_0 s_0+ \cdots + \bar{x}_i \bar{x}_n s_n=g(s_i),
\end{align*}
for all $i=0,\ldots,n$.
Therefore, $g$ is idempotent. Hence, by Proposition \ref{proposition:4.6}, $P_n:=\ker g$ is super projective and $P_n \oplus \Ima g \cong F$. Let $\alpha=\sum_{j=0}^n \bar{x}_j s_j$. Then,  
\begin{align*}
g(\alpha)&=g(\bar{x}_0 s_0 + \bar{x}_1 s_1 + \cdots + \bar{x}_n s_n)\\ &= \bar{x}_0 g(s_0) + \bar{x}_1 g(s_1) + \cdots + \bar{x}_n g(s_n) \\
&= \bar{x}_0 (\bar{x}_0^2 s_0 + \cdots + \bar{x}_0 \bar{x}_n s_n) + \cdots +  \bar{x}_n (\bar{x}_n \bar{x}_0 s_0 + \cdots + \bar{x}_n^2 s_n)\\
&= \bar{x}_0^3 s_0 + \cdots + \bar{x}_0^2 \bar{x}_n s_n + \cdots + \bar{x}_n^2 \bar{x}_0 s_0 + \cdots + \bar{x}_n^3 s_1 \\
&= \bar{x}_0(\bar{x}_0^2 s_0 + \cdots + \bar{x}_n^2 s_0) + \cdots + \bar{x}_n (\bar{x}_0^2 s_n + \cdots + \bar{x}_n^2 s_n) \\
&= \bar{x}_0 s_0 + \cdots + \bar{x}_n s_n=\alpha. 
\end{align*}
This shows that $y\in \Ima g$. On the other hand, any element in $\Ima g$ is generated by $\bar{x}_i(\sum_{j=1}^n \bar{x}_j s_j)=\bar{x}_i\alpha$, for $i= 0,1, \ldots ,n$. Therefore, $\alpha$ generates $\Ima g$. Hence $\Ima g$ is a free supermodule of rank $1$. So $P_n = \ker g$ has a free complement, that is, $P_n$ direct sum with finitely generated free supermodules is a free supermodule. Hence, $P_n$ is stably free. 	
\end{example}

Note that in the above example $P_n$ might be free for some $n$. For example when $R=\mathbb{R}_{S}$, and $n=1$, then $P_1$ is free of rank $1$, Indeed, $P_1=\Lambda(\bar{x}_1s_0-\bar{x}_0s_1)$.  In order to identify for what $n$, $P_n$ is projective but not free we need to relate $\Lambda$ to the tangent bundle of the supersphere. 
Let $S^{m,0}=\{(x_0,\ldots,x_m\in \mathbb{R}_{S[L]}^{m,0}\mid\sum_{i=0}^n x_i^2=1\}$.  
We can identify $\Lambda$ as a subring of the superring $G^\infty(S^{m,0},\mathbb{R}_{S[L]})$.
And $\widetilde{P}_m:=G^\infty(S^{m,0},\mathbb{R}_{S[L]})\otimes P_n\cong \Gamma(S^{m,0},TS^{m,0})$, where $\Gamma(S^{m,0}, TS^{m,0})$ is the supermodule of global vector fields.  Now $G^\infty(S^{m,0},\mathbb{R}_S)\otimes P_n$ is super projective but not free depends upon $(S^{m,0}, DeWitt, G^\infty)$ is parallellizable or not. 

By \cite[Example 2.3.1, p~426]{Bru1988} $(S^{m,0}, DeWitt, G^\infty)$ supermanifold.  Note that by this example, $G^\infty$ coincides with $\mathcal{G}$. Since $S^{m,0}$ is an even supermanifold, by \cite[Proposition 3.4]{Bru1988} $T_x(S^{m,0})$ is isomorphic to $T_xS^m$ (since the underlying smooth manifold for $S^{m,0}$ is $S^m$.)  By \cite[Corollary 2.1]{Bru1988} $T(S^{m,0})$ is locally free and we know that $TS^m$ is locally free.  Hence, both $TS^{m,0}$ and $TS^m$ are finitely presented.  Hence, by \cite[Remark 2.9]{Mor2013} $TS^{m,0}$ is isomorphic to $TS^m$. Thus $\widetilde{P}_m$ is stably free but not free for $m\neq 0,1,3,7$.

\subsection{Supersphere $S^{2,2}$ and super projective modules}

In this section, we describe projective modules which are constructed from vector bundles over the supersphere of dimension $(2,2)$.  This construction is by Landi.
We are using notations and results from Landi's paper.  We are not describing them here as we are not using them in the actual calculations. 

Let $\mathbb{R}_{S[2r]}$ be a superring as defined in Example \ref{example:2.5}. We denote by $\mathbb{C}$ the field of complex numbers. Let $\mathbb{C}_{S[2r]}=\mathbb{R}_{S[2r]}\otimes \mathbb{C}$.  Clearly $\mathbb{C}_{S[2r]}$ is a superring.  And let 
${}^\Diamond:\mathbb{C}_{S[2r]}\to \mathbb{C}_{S[2r]}$ be a graded involution, \cite[page 51]{BBL1990} or \cite[equation (3.2)]{Lan2001}. Consider the $\mathbb{R}_{S[2r]}^{m,n}$ with the DeWitt topology, defined as in Section \ref{subsection:supermanifold} with supersmooth functions are $GH^\infty$.  We described $G^\infty$. Supersmooth functions, $GH^\infty$ are defined smilarly. See \cite[Chapter 4]{Rog2007} and \cite{BBL1990}.

A \emph{supersphere $S^{2,2}$} is defined as,
\begin{align*}
   S^{2,2} &=\{ (x_1, x_2, x_3, \xi_1, \xi_2)\in \mathbb{C}_{S[2r]}^{3,2}\mid \Sigma_{i = 1}^{3} x_i^2 + 2\xi_1 \xi_2=1,\\
   & \hspace*{5cm} \text{ and } x_i^\Diamond=x_i,\text{ for all } i=1,2,3 \}.
\end{align*}
This supersphere is isomorphic to $S^{2,0}\times \mathbb{R}_{S[2r]}^{0,2}$, where $S^{2,0}=\{(x,y,z)\in\mathbb{R}_{S[2r]}^{3,0}\vert x^2+y^2+z^2=1 \}$, \cite[Proposition 5.1]{BBL1990}.
Further the same proposition proves that $S^{2,2}$ is isomorphic to the quotient of Lie superalgebra $UOSP(1,2) / \mathcal{U}(1)$.  Here $UOSP(1,2)$ is a subalgebra of $\mathbb{M}_{3\times 3}(\mathbb{C}_{S[2r]})$, and elements of $UOSP(1,2)$ are parameterized by three variables. So, \[UOSP(1,2)=\{s(a,b,\eta)\vert a,b\in (\mathbb{C}_{S[2r]})_0, \eta\in (\mathbb{C}_{S[2r]})_1,\text{ and }aa^\Diamond+bb^\Diamond=1\},\] where $s(a,b,\eta)$ is as given in \cite[Equation 5.2]{BBL1990} or \cite[Equation 3.9]{Lan2001}.  And  $\mathcal{U}(1)=\{w\in(\mathbb{C}_{S[2r]})_0\vert ww^\Diamond=1\}$.

For given $n\in\mathbb{N}$, consider a map $\langle \psi_n\vert:UOSP(1,2)\to (\mathbb{C}_{S[2r]})_0^{n+1}\times (\mathbb{C}_{S[2r]})_1^n$, given by 

\begin{align*}
\langle \psi_n\vert\bigl(s(a,b,\eta)\bigr) &=\left((1-\frac{1}{8} \eta \eta^{\Diamond})\left((a^\Diamond)^n,\ldots, \sqrt{\binom nk} (a^\Diamond)^{n-k} (b^\Diamond)^k,\ldots,(b^\Diamond)^n\right),\right .\\
&\left .\frac{1}{2}\eta^\Diamond\left((a^\Diamond)^n,\ldots,\sqrt{\binom {n-1}{k}}(a^\Diamond)^{n-1-k}(b^\Diamond)^k, \ldots , (b^\Diamond)^{n-1}\right)\right),
\end{align*}
where $\binom nk = \frac{n!}{k!(n-k)!}$. See \cite[Subsection 4.2.1]{Lan2001} for details.  It is easy to check that, $\langle \psi_n\vert \psi_n\rangle$ is a constant $1$ function, where $\vert\psi_{n}\rangle= \langle \psi_{n}\vert^t$. Let $\mathbb{M}_{n+1,n}$ be the set of even matrices.
And, let the projector $p_n:UOSP(1,2)\to\mathbb{M}_{n+1,n}(\mathbb{C}_{S[2r]})$ defined by 
$p_n (s(a,b,\eta))= \bigl(\vert\psi_n\rangle \langle \psi_n \vert\bigr)\bigl(s(a,b,\eta)\bigr)$. Using this map we want to construct a super projective module.  

Before that, we need to define superring over which our super projective module will be constructed. Let $\mathcal{A}_{\mathbb{C}_{S[2r]}}$ (respectively $\mathcal{B}_{\mathbb{C}_{S[2r]}}$) denote the superring of global supersmooth functions on $S^{2,2}$ (respectively, $UOSP(1,2)$)which are $\mathbb{C}_{S[2r]}$-valued. In general, we define $\mathcal{A}_{\mathbb{C}_{S[2r]}}^{m,n}=GH^\infty(S^{2,2}, \mathbb{C}_{S[2r]}^{m,n})$ (respectively, $\mathcal{B}_{\mathbb{C}_{S[2r]}}^{m,n}={GH}^\infty(UOSP(1,2),\mathbb{C}_{S[2r]}^{m,n})$).

Let $\pi_n:\mathcal{B}_{\mathbb{C}_{S[2r]}}^{n+1,n}\to \mathcal{B}_{\mathbb{C}_{S[2r]}}^{n+1,n}$, be given by 
\begin{align*}
 \pi_n\begin{pmatrix}
\bar{f}; \bar{g}
\end{pmatrix}^t:UOSP(1,2)&\to \mathbb{C}_{S[2r]}^{n+1,n},\\ 
(a,b,\eta)& \mapsto \bigl(1- \frac{1}{8} \eta \eta^{\Diamond}\bigl)\sum_{k = 0}^{n} \sqrt{\binom{n}{k}}(a^{\Diamond})^{n-k} (b^{\Diamond})^k g_k(a,b,\eta)\\
& \hspace*{2 em} + \frac{1}{2} \eta^{\Diamond} \sum_{j = 0}^{n-1} \sqrt{\binom{n-1}{j}}(a^{\Diamond})^{n-1-j} (b^{\Diamond})^j f_j(a,b,\eta).
\end{align*}
where $(\bar{f};\bar{g})=\begin{matrix} (f_0 & f_1 &\cdots & f_n; & g_1 & \cdots& g_n)\end{matrix}$
By \cite[equation 4.16 subsection 4.2.1]{Lan2001} it is clear that $\pi_n^2=\pi_n$. Now using the quotient map from $UOSP(1,2)$ to $UOSP(1,2)/\mathcal{U}(1)$ and an isomorphism between $S^{2,2}$ and $UOSP(1,2)/\mathcal{U}(1)$, we get associated map $\widetilde{\pi_n}:
\mathcal{A}_{\mathbb{C}_{S[2r]}}^{n+1,n}\to \mathcal{A}_{\mathbb{C}_{S[2r]}}^{n+1,n}$, such that $\widetilde{\pi_n}^2=\widetilde{\pi_n}$.  Now by Proposition \ref{proposition:4.6} we get $\Ima \widetilde{\pi_n}$ is a super projective module. 
\vspace{.2cm}

\noindent{\bf Acknowledgments:}\\

The third author is thankful to the faculty and the administrative unit of the institute of mathematical sciences (IMSc) for their warm hospitality during the preparation of the paper.


\bibliography{sn-bibliography}%

\end{document}